\documentclass{amsart}
\usepackage{amsmath,amsthm,amsfonts, amssymb,amscd}
\usepackage[all]{xy}

\newtheorem{theorem}{Theorem}[section]
\newtheorem{lemma}[theorem]{Lemma}
\newtheorem{remark}[theorem]{Remark}
\newtheorem{corollary}[theorem]{Corollary}
\newtheorem{proposition}[theorem]{Proposition}

\newcommand{\squig}{\rightsquigarrow}

\def\ld{\mathord{\backslash}}
\def\rd{\mathord{/}}

\begin{document}
\title[On a New Construction of Pseudo BL-Algebras]{On a New Construction of Pseudo BL-Algebras}
\author[Anatolij Dvure\v{c}enskij]{Anatolij Dvure\v censkij$^{1,2}$}
\date{}%
\maketitle
\begin{center}  \footnote{


This work was supported by  the Slovak Research and Development Agency under contract APVV-0178-11,  grant VEGA No. 2/0059/12 SAV, and
CZ.1.07/2.3.00/20.0051.
 }
Mathematical Institute,  Slovak Academy of Sciences,\\
\v Stef\'anikova 49, SK-814 73 Bratislava, Slovakia\\
$^2$ Depart. Algebra  Geom.,  Palack\'{y} University\\
17. listopadu 12, CZ-771 46 Olomouc, Czech Republic\\

E-mail: {\tt dvurecen@mat.savba.sk}
\end{center}

\begin{abstract}
We present a new construction of a class pseudo BL-algebras, called kite pseudo BL-algebras. We start with a basic pseudo hoop $A$. Using two injective mappings from one set, $J$, into the second one, $I$, and with an identical copy $\overline A$ with the reverse order we construct a pseudo BL-algebra where the lower part is of the form $(\overline A)^J$ and the upper one is $A^I$. Starting with a basic commutative hoop we can obtain even a non-commutative pseudo BL-algebra or a pseudo MV-algebra, or an algebra with non-commuting negations. We describe the construction, subdirect irreducible kite pseudo BL-algebras and their classification.
\end{abstract}

\section{Introduction}

The authors of \cite{GLP} presented pseudo hoops which were originally introduced by Bosbach in \cite{Bos1,Bos2} under
the name ``residuated integral monoids''. Pseudo hoops generalize pseudo BL-algebras \cite{DGI1,DGI2}, pseudo MV-algebras
\cite{GeIo} (= GMV-algebras \cite{Rac}). Pseudo MV-algebras are always intervals in unital $\ell$-groups \cite{Dvu1}, i.e. of the form $[0, u]$, where $u$ is a strong unit. These structures can also be studied in the frame of integral residuated lattices \cite{GaTs}, as every pseudo hoop is a meet-semilattice ordered residuated, integral and divisible monoid.

Pseudo BL-algebras are a non-commutative generalization of BL-algebras introduced in \cite{Haj1} as an algebraic semantics of basic fuzzy logic. Non-commutative versions with a logical background can be found in \cite{Haj2, Haj3}. Pseudo BL-algebras are now intensively studied by many authors also as a special class of pseudo MTL-algebras generalizing MTL-algebras introduced in \cite{EsGo}. In \cite{Dvu2} it was shown that every linear pseudo hoop is an ordinal sum of linear Wajsberg pseudo hoops which have close connections with cones of linear $\ell$-groups; this result generalizes analogous one from \cite{AgMo} proved for linear hoops.

Every pseudo BL-algebra gives two negations, left and right one. In \cite[Prob 3.21]{DGI2}, it was posed an open problem whether every pseudo BL-algebra is good, i.e. whether these two negations commute. Every pseudo MV-algebras or every linear or representable pseudo BL-algebra is always good, \cite{Dvu2}. The problem was answered in negative in \cite{DGK} where an example originally constructed in \cite{JiMo} was used. This example has initiated an interesting research on so-called kite pseudo BL-algebras, see \cite{DvKo}. The example by Jipsen and Montagna resembles the kite shape therefore, it gave the name for kite pseudo BL-algebras. In \cite{DvKo}, a starting point was an $\ell$-group with two injections $\lambda,\rho: J \to I$ and with  the positive and negative cones. The algebra with a lower part $(G^+)^J$ and an upper part $(G^+)^I$ was called a kite and it is always a pseudo BL-algebra, in special cases a non good one or even a pseudo MV-algebra.
If we take into account that the negative cone of each $\ell$-group can be viewed as a special kind of a basic pseudo BL-algebra, in the present paper we introduced the kite pseudo BL-algebras starting with basic pseudo hoops. We modified the construction from \cite{DvKo} to this more general situation because there are basic pseudo hoops that are not negative cones of any $\ell$-group.
In this paper, we extend results from \cite{DvKo}. The ideas of the proofs resemble the original proofs but the present proofs have to be modified to a completely new situation. The basic results are connected with subdirectly irreducible kites, their characterizations and classifications. In particular, we have non-linear kites that are subdirectly irreducible.

The present paper gives a new large supply of interesting examples of pseudo BL-algebras, and we hope that it opens a new door into theory of pseudo BL-algebras.

The paper is organized as follows. Section 2 gathers the basic notions and results on pseudo hoops, pseudo BL-algebras and pseudo MV-algebras. Section 3 presents the basic construction of a kite pseudo BL-algebra starting with a basic pseudo hoop. The resulting algebra can be even a pseudo MV-algebra or a pseudo BL-algebra that is not good. Section 4 characterizes subdirectly irreducible kites and we present their classification. In particular, we show that every kite is subdirectly embeddable into a product of subdirectly irreducible kites. Some examples and open problems are presented, too.

\section{Elements of Pseudo Hoops and Pseudo BL-Algebras }

According to \cite{GLP}, an algebra $\mathbf A=(A;\cdot,\ld,\rd,1)$ of type
$\langle 2,2,2,0\rangle$ is a {\it pseudo hoop} if the following holds, for all $a,b \in A$
\begin{enumerate}
\item[(i)] $a \cdot 1 = 1\cdot a=a$;
\item[(ii)] $a\ld a=1=a\rd a$;
\item[(iii)] $c\rd (a\cdot b) = (c\rd b)\rd a$;
\item[(iv)] $(a\cdot b)\ld c= b \ld(a\ld c)$;
\item[(v)] $(b\rd a)\cdot a= (a\rd b)\cdot a= a \cdot (a\ld b)=b\cdot (b\ld a)$.
\end{enumerate}

This means that $(A;\cdot,\ld,\rd, 1)$ is a residuated monoid, (i.e. $\cdot$ is associative, with unit element $1,$ and $x\cdot y \le z$ iff $y \le x\ld z $ iff $x \le z\rd y$ for all $x,y,z \in A$). For simplicity, we will write $xy$ instead of $x \cdot y$. We note that in \cite{GLP} there are used arrows $\to$ and $\squig$ which are connected with divisions as follows: $y\to z=z\rd y$ and $y \squig z = y\ld z$. A pseudo hoop $\mathbf A$ is non-trivial if $A\ne \{1\}$.

The operations $\ld$ and $\rd$ are called
\emph{left division} (or \emph{right residuation}) and
\emph{right division} (or \emph{left residuation}), respectively.
We will assume that $\cdot$ has higher binding priority than $\ld$ and $\rd$,  which in turn bind stronger than the lattice connectives.

Then the relation $\le$ defined by $a\le b$ iff $a\ld b =1$ (iff $b\rd a=1)$ is a partial order on $A,$ in addition $a\le b$ iff there is $c \in A$ such that $a=c\cdot b$, and $A$ is a $\wedge$-lattice because $a\wedge b = (b\rd a)\cdot a= (a\rd b)\cdot a= a \cdot (a\ld b)=b\cdot (b\ld a)$ with the top element $1$.

By \cite[Rem 2.3]{GLP}, the operation $\cdot$ is commutative iff $\ld=\rd$, and in such a case, $\mathbf A$ is said to be a {\it hoop}.

Let $\mathbf G=(G;\cdot,^{-1},e,\le)$ be an $\ell$-group (= lattice ordered group) written multiplicatively with an inversion $^{-1}$ and the identity element $e$, equipped with a lattice order $\le$ such that $a\le b$ entails $cad\le cbd$ for all $c,d \in G$. We denote $G^+=\{g \in G: e\le g\}$ and $G^-=\{g \in G: g \le e\}$  the positive and negative cone, respectively, of $G$. An element $ u \in G^+$ is said to be a {\it strong unit} if given $g \in G$ there is an integer $n\ge 1$ such that $g \le u^n$. A pair $(\mathbf G,u)$, where $\mathbf G$ is an $\ell$-group with a fixed strong unit $u$, is said to be a {\it unital $\ell$-group}.

For example, let $\mathbf G=(G;\cdot,^{-1},e,\le)$ be an $\ell$-group. If we endow the negative cone $G^-$ with binary operations $a\cdot b = ab$, $a\ld b = (a^{-1}b)\wedge e$ and $a\rd b=(ab^{-1})\wedge e,$ then $(G^-;\cdot,\ld,\rd,e)$ is an example of a pseudo hoop.

A pseudo BL-algebra is a special kind of a pseudo hoop with an additional fixed element $0$ introduced in \cite{DGI1, DGI2}; we present an equivalent definition, see e.g. \cite{DvKo}:  An algebra $\mathbf A=(A;\cdot,\ld,\rd,\wedge,\vee, 0,1)$ of type $\langle 2,2,2,2,2,0,0\rangle $ is a {\it pseudo BL-algebra} if

\begin{enumerate}
\item[(i)] $(A;\wedge,\vee,0,1)$ is a bounded lattice;
\item[(ii)] $(A;\cdot,\wedge,\vee,1)$ is a residuated monoid;
\item[(iii)] $x(x\ld (x\wedge y)) = x\wedge y = ((x\wedge y)\rd x)x$;
\item[(iv)] $x(x\ld y) = x\wedge y = (y\rd x)x$;\quad (divisibility)
\item[(v)]$x\ld y \vee y\ld x = 1 = y\rd x \vee x\rd y$ \quad (prelinearity).
\end{enumerate}

We define two negations: $x^-=0\rd x$ and $x^\sim =x\ld 0$ of $x \in A,$ and we say that a pseudo BL-algebra is {\it good} if $x^{-\sim}=x^{\sim-}$ for all $x \in A$. A {\it pseudo MV-algebra} is a good BL-algebra satisfying $x^{-\sim}=x=x^{\sim-}$ for all $x \in A$. For more information about pseudo MV-algebras (equivalently, for GMV-algebras), see \cite{GeIo,Rac}.

For example, if $(\mathbf G,u)$ is a unital $\ell$-group, we endow the set  $G[0,u]:=\{g \in G \colon e\le g \le u\}$ with $a\cdot b= (au^{-1}b)\vee e$, $b\rd a= (ba^{-1}u)\wedge u$, $a\ld b= (ua^{-1}b)\wedge u$, then $\Gamma(\mathbf G,u)=(G[0,u];\cdot,\ld,\rd,\wedge ,\vee, e,u)$ is an example of pseudo MV-algebras. Due to \cite{Dvu1}, for every pseudo MV-algebra $\mathbf A$, there is a unique (up to isomorphism of unital $\ell$-groups) unital $\ell$-group $(\mathbf G,u)$ such that $\mathbf A \cong \Gamma(\mathbf G,u)$, and $(\mathbf G,u) \mapsto \Gamma(\mathbf G,u)$ defines a categorical equivalence between the category of unital $\ell$-groups and the variety of pseudo MV-algebras.

It was an open problem in \cite[Prob 3.21]{DGI2} whether every pseudo BL-algebra is good, and it was answered in \cite{DGK} in negative. Thus the variety of good pseudo BL-algebras is a proper subvariety of the variety of pseudo BL-algebras.

According to \cite{GLP}, we say that a pseudo hoop $\mathbf A=(A;\cdot,\ld,\rd,1)$ is {\it basic} if it satisfies the following two conditions:
\begin{enumerate}
\item[(i)] $c \rd(b\rd a)\le c\rd (c \rd (a\rd b))$;
\item[(ii)] $(a\ld b)\ld c \le ((b\ld a)\ld c)\ld c$.
\end{enumerate}

Then $A$ is a distributive lattice such that $((b\rd a)\ld b) \wedge (((a\rd b)\ld a) =a\vee b = (b\rd (a\ld b)\wedge (a\rd (b\ld a))$ and by \cite[Prop 4.10]{GLP}, the variety of bounded basic pseudo hoops is termwise equivalent to the variety of pseudo BL-algebras. For example, $(G^-;\cdot,\ld,\rd,e)$ is a basic pseudo hoop.

It is straightforward to verify that any linearly ordered pseudo hoop and hence any representable pseudo hoop is basic. We notice that not every pseudo hoop is basic, \cite[Rem 5.10]{GLP} (take $A_1$ a nonlinear pseudo hoop and
$A_2 = 2^1$, then the ordinal sum $A = A_1\bigoplus A_2$ ($\bigoplus$ denotes ordinal sum) is not basic), and not all pseudo BL-algebras are representable.

We note that according to \cite{GLP}, a special basic hoop is a {\it Wajsberg} pseudo hoop $\mathbf A=(A; \cdot, \ld,\rd,1)$ such that the following identities hold
\begin{enumerate}
\item[(W1)] $(b\rd a)\ld b= (a\rd b)\ld a$;
\item[(W2)] $b\rd (a\ld b)=a\rd (b\ld a).$
\end{enumerate}

We note that a subset $F$ of $A$ is said to be a {\it filter} of $\mathbf A$ if (i) $a,b \in F$ implies $ab\in F$, and (ii) $a\le b$, $a\in F,$ $b \in A$ imply $b \in F$.  Any filter of $\mathbf A$ is a pseudo subhoop of $\mathbf A$. A filter $F$ that is a proper subset of $A$ is {\it maximal} if there is no other proper filter properly containing $F$.

A filter $F$ is {\it normal} if, for all $a,b \in A$, $b\rd a \in F$ iff $a\ld b\in F$. Equivalently, $a\cdot F=\{ab: b \in F\}=F\cdot a =\{ca: c \in F\}.$  The singleton $\{1\}$ and $A$ are always normal filters of $\mathbf A$. According to \cite[Prop 3.15]{GLP}, there is a one-to-one correspondence between normal ideals of a pseudo hoop $\mathbf A$ and congruences on $\mathbf A$: If $F$ is a normal ideal,
then $\approx_F$, defined by $a\approx_F b$ iff $a\ld b, b\ld a\in F$, is a congruence; conversely, if $\approx$ is a congruence, then $F_\approx =\{a\in A: a\approx 1\}$ is a normal filter of $\mathbf A$. Hence, to prove that a pseudo hoop $\mathbf A$ is subdirectly irreducible, it is equivalent to prove that it is either trivial or it contains the least non-trivial normal filter $F\ne \{1\}$.

\section{Construction of Kite Pseudo BL-algebras}

Motivating by \cite{DvKo}, we present a new construction of kite pseudo BL-algebras starting now with basic pseudo hoops instead of $\ell$-groups. We show that in particular cases we can obtain pseudo MV-algebras or pseudo BL-algebras which are not good.

Take two index sets $J$ and $I$ such that $|J|\le |I|$ and let $\lambda,\rho: J \to I$ be injections. Notice that the case $J=I$ is not excluded.

Suppose that $\mathbf A = (A;\cdot,\ld,\rd,1)$ is a basic pseudo hoop. We denote by $\overline{A}$ an identical copy of $A$ whose elements are of the form $\bar a$ for all $a\in A$, that is $\overline{A} :=\{\bar a\colon a \in A\}$. We assume that $\bar a= \bar b$ iff $a=b$.

We define two special elements $0 = \bar 1^J:=\langle \bar 1_j\colon j \in J\rangle$ and $1=  1^I:= \langle 1_i\colon i \in I\rangle$, where $1_i=1=1_j$ for all $j \in J$ and $i \in I$. The elements of $A^I$ will be denoted by $\langle a_i\colon i \in I\rangle$ and ones of $(\overline A)^J$ by $\langle \bar f_j\colon j \in J\rangle$, where $a_i,f_j \in A$ for all $j\in J$ and $i\in I$. We set $(\overline{A})^J \uplus A^I$ as a disjoint union with a lower part $(\overline A)^J$ and an upper part $A^I$. We order $A^I$ by the coordinate ordering and $\langle \bar f_j\colon j \in J\rangle \le \langle \bar g_j\colon j \in J\rangle$ iff $g_j \le f_j$ for all $j \in J,$ and let $x\le y$ for all $x \in \overline A^J$ and $y \in A^I$. So that the universe $(\overline A)^J \uplus A^I$ is a bounded lattice.

We define a multiplication on $(\overline A)^J \uplus A^I$ as follows

\begin{align*}
\langle a_i\colon i\in I\rangle\cdot\langle b_i\colon i\in I\rangle &=
  \langle a_ib_i\colon i\in I\rangle\\
\langle  a_i\colon i\in I\rangle\cdot\langle \bar f_j\colon j\in J\rangle &=
  \langle \overline{f_j \rd a_{\lambda(j)}}\colon j\in J\rangle\\
\langle \bar f_j\colon j\in J\rangle\cdot\langle  a_i\colon i\in I\rangle &=
  \langle \overline{a_{\rho(j)}\ld f_j}\colon j\in J\rangle\\
\langle \bar f_j\colon j\in J\rangle\cdot\langle \bar g_j\colon j\in J\rangle &=
  \langle \bar 1_j\colon j\in J\rangle = 0.\\
\end{align*}

Divisions, $\ld$ and $\rd$, are defined by
\begin{align*}
\langle a_i\colon i\in I\rangle\ld\langle b_i\colon i\in I\rangle &=
 \langle a_i\ld b_i\colon i\in I\rangle\\
\langle b_i\colon i\in I\rangle\rd\langle a_i\colon i\in I\rangle &=
 \langle b_i\rd a_i\colon i\in I\rangle\\
\langle a_i\colon i\in I\rangle\ld\langle \bar f_j\colon j\in J\rangle &=
 \langle \overline{f_ja_{\lambda(j)}}\colon j\in J\rangle\\
\langle \bar f_j\colon j\in J\rangle\rd\langle a_i\colon i\in I\rangle &=
 \langle \overline{a_{\rho(j)}f_j}\colon j\in J\rangle\\
\langle \bar f_j\colon j\in J\rangle\ld\langle \bar g_j\colon j\in J\rangle &= \langle a_i\colon i\in I\rangle,\\
\text{ where } a_i &=\begin{cases}
f_{\rho^{-1}(i)}\rd g_{\rho^{-1}(i)} &
\text{ if } \rho^{-1}(i) \text{ is defined}\\
1 & \text{ otherwise,}
\end{cases}\\
\langle \bar g_j\colon j\in J\rangle\rd\langle \bar f_j\colon j\in J\rangle &= \langle b_i\colon i\in I\rangle,\\
\text{ where } b_i &=\begin{cases}
g_{\lambda^{-1}(i)}\ld f_{\lambda^{-1}(i)} &
\text{ if } \lambda^{-1}(i) \text{ is defined}\\
1 & \text{ otherwise},
\end{cases}\\
\langle a_i\colon i\in I\rangle\rd \langle \bar f_j\colon j\in
J\rangle &=(1)^I= \langle \bar f_j\colon j\in J\rangle \ld \langle
a_i\colon i\in I\rangle.
\end{align*}

Without any unambiguous problems, we use the same notations for multiplications and divisions in the basic pseudo hoop as well as in  $(\overline A)^J \uplus A^I$.

\begin{theorem}\label{th:3.1}
Let $\mathbf A=(A;\cdot,\ld,\rd,1)$ be a basic pseudo hoop. Let $\lambda,\rho:J \to I$ be injections. Then $((\overline A)^J \uplus A^I; \cdot,\ld,\rd,\wedge,\vee, 0,1)$, where $\cdot,\ld,\rd,\wedge,\vee, 0,1$ were defined above, is a pseudo BL-algebra.
\end{theorem}

\begin{proof}
As it was mentioned, $(\overline A)^J \uplus A^I$ is a bounded lattice with $0=\langle \bar 1_j\colon j \in J\rangle$ and $1=\langle 1_i\colon i \in I\rangle$. To show that multiplication is associative, we observe that if  triples are from $A^I$, then multiplication is associative, the same is true if triples are from $(\overline A)^J$. If triples involve at least two elements from $(\overline A)^J$, they associate because both products equal $0$. The remaining inspection is to verify cases when one element of the triple is from $(\overline A)^J$ and two from $A^I.$ One such case is the following; here we use the fact that, for $x,y,z \in A,$ we have $x\ld (y\rd z)=(x\ld y)\rd z,$ see e.g. \cite[Lem 2.1(vii)]{GaTs}:

\begin{align*}
(\langle a_i\colon i\in I\rangle\cdot\langle \bar f_j\colon j\in J\rangle)\cdot\langle b_i\colon i\in I\rangle &=
\langle \overline{f_j \rd a_{\lambda(j)}}\colon j\in J\rangle\cdot\langle b_i\colon i\in I\rangle\\
&= \langle\overline{b_{\rho(j)}\ld (f_j \rd a_{\lambda(j)})}\colon j\in J\rangle\\
&= \langle\overline{(b_{\rho(j)}\ld f_j) \rd a_{\lambda(j)})}\colon j\in J\rangle\\
&= \langle a_i\colon i\in I\rangle\cdot\langle \overline{b_{\rho(j)}\ld f_j}\colon j\in J\rangle\\
&= \langle a_i\colon i\in I\rangle\cdot(\langle \bar f_j\colon j\rangle\cdot\langle b_i\colon i\rangle).
\end{align*}

The rest follows similar reasonings. It is clear that $1=\langle  1_i\colon i \in I\rangle$ is a neutral element for multiplication.

It is straightforward to verify that $\ld$ and $\rd$ are residuation operators.

Now we verify divisibility. We show the following case

\begin{align*}
\langle \bar f_j\colon j\in J\rangle\cdot(\langle \bar f_j\colon j\in J\rangle\ld\langle \bar g_j\colon j\in J\rangle)
&= \langle \bar f_j\colon j\in J\rangle\cdot\langle a_i\colon i\in I\rangle \\
&= \langle \overline{a_{\rho(j)}\ld f_j}\colon j\in J\rangle,
\end{align*}
where
$$
a_ican appear
 =\begin{cases}
f_{\rho^{-1}(i)}\rd g_{\rho^{-1}(i)} &
\text{ if } \rho^{-1}(i) \text{ is defined}\\
1 & \text{ otherwise}
\end{cases}
$$
but, observe that $\rho^{-1}(\rho(j))$ is always defined and equals $j$,
so calculating further we obtain $a_{\rho(j)}=f_j\rd g_j$ and
$\langle \overline{a_{\rho(j)}\ld f_j}\colon j\in J\rangle = \langle \overline{(f_j\rd g_j)\ld f_j}\colon j \in J\rangle = \langle \overline{f_j\vee g_j}\colon j \in J\rangle = \langle \bar f_j\colon j \in J\rangle \wedge \langle \bar g_j\colon  j \in J\rangle$. In the same way we proceed for the dual identity.

Now we show prelinearity.  If both elements $x,y$ are from $A^I$, prelinearity holds due to coordinatewise calculations and due to prelinearity of the basic pseudo hoop \cite[Lem 4.5]{GLP}. If $x \in A^I$ and $y\in (\overline A)^J$ or $y \in A^I$ and $x \in  (\overline A)^J,$ prelinearity holds trivially. We have to verify only the following case

$$
\langle \bar f_j\colon j\in J\rangle\ld\langle \bar g_j\colon j\in J\rangle\vee
\langle \bar g_j\colon j\in J\rangle\ld\langle \bar f_j\colon j\in J\rangle =\langle a_i\colon i \in I\rangle \vee \langle b_i\colon i \in I\rangle
$$
which yields two cases: (1) if $\rho^{-1}(i)$ is defined, we have
\begin{align*}
&  f_{\rho^{-1}(i)}\rd  g_{\rho^{-1}(i)}
\vee  g_{\rho^{-1}(i)}\rd f_{\rho^{-1}(i)}=1_i\\
\end{align*}
and (2) if $\rho^{-1}(i)$ is not defined, we have
$$
a_i \vee b_i = 1\vee 1 =1_i.
$$

Summarizing all steps, we see that the algebra in question is a pseudo BL-algebra as was claimed.
\end{proof}

If the set $J$ is a singleton, $|I|=2$ and $A=\mathbb Z^-$, where $\mathbb Z$ is the group of integers, the universe of the corresponding pseudo BL-algebra resembles the shape of a kite, and similarly as in \cite{DvKo}, the corresponding algebra $K_{I,J}^{\lambda,\rho}(\mathbf{A})= ((\overline A)^J \uplus A^I; \cdot,\ld,\rd,\wedge,\vee, 0,1)$ is said to be a {\it kite pseudo BL-algebra} corresponding to the basic pseudo hoop $\mathbf A$, or simply a kite. The kite pseudo BL-algebra corresponding to an $\ell$-group $\mathbf G$ studied in \cite{DvKo} is a special case of our situation when the basic pseudo hoop $\mathbf A$ is the basic pseudo hoop corresponding to the negative cone $G^-$, i.e., $\mathbf A=(G^-;\cdot,\ld,\rd,e)$. Therefore, using these negative cones, sets $J,I$ and injective mappings $\lambda,\rho:J \to I$, we can obtain a much larger supply of kite pseudo BL-algebras than the class of kites using only negative cones.

We obtain a special example of kites if $A=\{1\},$ then $K^{\lambda,\rho}_{I,J}(\mathbf A)$ is a two-element Boolean algebra. The same is true for the kite $K^{\emptyset,\emptyset}_{\emptyset,\emptyset}(\mathbf A)$ for any basic pseudo hoop $\mathbf A$.

Now we show that the kite pseudo BL-algebra can be also a pseudo MV-algebra.

\begin{proposition}\label{pr:3.2}
Let $\mathbf{A}$ be a basic pseudo hoop, and
suppose $|I| = |J|$ and $\lambda,\rho$ are bijections. Then
$K_{I,J}^{\lambda,\rho}(\mathbf{A})$
is a pseudo MV-algebra.
\end{proposition}

\begin{proof}
To show that $K_{I,J}^{\lambda,\rho}(\mathbf{A})$ is a pseudo MV-algebra, it is necessary to show that $x^{-\sim}=x=x^{\sim-}$ for all $x \in K_{I,J}^{\lambda,\rho}(\mathbf{A})$.

Check
\begin{align*}
(\langle a_i \colon i \in I\rangle)^{-\sim} &= (\langle \bar 1_j\colon j \in J\rangle \rd \langle a_i \colon i \in I\rangle) \ld \langle \bar 1_j\colon j \in J\rangle \\&= \langle \bar a_{\rho(j)} \colon i \in I\rangle \ld \langle \bar 1_j\colon j \in J\rangle = \langle a_i\rd 1_i \colon i \in I\rangle \\&=
\langle a_i \colon i \in I\rangle.
\end{align*}

\begin{align*}
(\langle a_i \colon i \in I\rangle)^{\sim-} &= \langle \bar 1_j\colon j \in J\rangle \rd (\langle a_i \colon i \in I\rangle \ld \langle \bar 1_j\colon j \in J\rangle) \\&= \langle \bar 1_j\colon j \in J\rangle\rd \langle \bar a_{\lambda(j)}\colon j \in J\rangle\\
&= \langle 1_i\ld a_i\colon i \in I\rangle = \langle a_i \colon i \in I\rangle.
\end{align*}

In the same way we proceed with $x = \langle \bar f_j \colon j \in J\rangle$.
\end{proof}

Now we characterize kites which are good pseudo BL-algebras and pseudo MV-algebras, respectively. As it was already mentioned, there was an open problem in \cite{DGI2} whether every pseudo BL-algebra is good, and it was answered negatively in \cite{DGK}. Here we present another class of pseudo BL-algebras that are not good.

\begin{theorem}\label{th:3.3}
Let $\mathbf{A}$ be a non-trivial basic pseudo hoop and
$K_{I,J}^{\lambda,\rho}(\mathbf{A})$ a kite.
\begin{enumerate}
\item
$K_{I,J}^{\lambda,\rho}(\mathbf{A})$ is good if and only if
$\lambda(J)=\rho(J)$.
\item $K_{I,J}^{\lambda,\rho}(\mathbf{A})$ is a pseudo MV-algebra
if and only if $\lambda(J)=I=\rho(J)$.
\end{enumerate}
\end{theorem}

\begin{proof}
(1) Let $x = \langle a_i \colon i \in I\rangle$.

In addition, $ x^{-\sim}=
\langle x_i\colon i\in I\rangle, $ where
$$
x_i=\begin{cases} a_i &
\text{ if } \lambda^{-1}(i) \text{ is defined}\\
1_i & \text{ otherwise},
\end{cases}
$$
and $x^{\sim-}= \langle y_i\colon i\in I\rangle$, where
$$
y_i=\begin{cases} a_i &
\text{ if } \rho^{-1}(i) \text{ is defined}\\
1_i & \text{ otherwise}.
\end{cases}
$$
Now, if $x = \langle \bar f_j\colon j\in J\rangle$, we have
$x^{-\sim} = \langle \bar g_j\colon j\in J\rangle$, where
$$
g_j =
\begin{cases} f_j & \text{ if } \rho^{-1}(i) \text{ is defined}\\
1 & \text{ otherwise},
\end{cases}
$$
and $ x^{\sim-} = \langle h_j\colon j\in J\rangle$, where
$$
h_j =
\begin{cases} f_j & \text{ if } \lambda^{-1}(i) \text{ is defined}\\
1 & \text{ otherwise}.
\end{cases}
$$
Hence, if $\lambda(J)=\rho(J),$ the kite pseudo BL-algebra
$K_{I,J}^{\lambda,\rho}(\mathbf{A})$ is good.

Conversely, assume that the kite pseudo BL-algebra
$K_{I,J}^{\lambda,\rho}(\mathbf{A})$ is good, and let $\lambda(J)\ne
\rho(J).$ Take $x = \langle a_i \colon i \in I\rangle$  where each $a_i \ne
1.$ There is an $i\in I$ such that either $\lambda^{-1}(i)$ or
$\rho^{-1}(i)$ is not defined. Equivalently, $x_i= 1$ and
$x_i=a_i$ or $y_i=1$ and $y_i=a_i.$
Hence, $\lambda(J)=\rho(J).$

(2) If $\lambda(J)=I=\rho(J)$, then
$K_{I,J}^{\lambda,\rho}(\mathbf{A})$ is a pseudo MV-algebra by Proposition
\ref{pr:3.2}. Conversely, let the kite pseudo BL-algebra $K_{I,J}^{\lambda,\rho}(\mathbf{A})$ be a pseudo MV-algebra.
Since every pseudo MV-algebra is good, by the first part of the
present proof, we have $\lambda(J)=\rho(J)$. Now assume that there is an $i\in I\setminus \lambda(J)$. Then both
$\lambda^{-1}(i)$ and $\rho^{-1}(i)$ are not defined, whence
$x_i=1 = y_i \ne a_i$ which contradicts the property
$x^{-\sim}=x=x^{\sim-}$. Therefore, $\lambda(J)=I=\rho(J)$.
\end{proof}

\section{Subdirectly Irreducible Kite Pseudo BL-algebras}

In this section we describe and classify subdirectly irreducible kite pseudo BL-algebras. We show that subdirectly irreducible algebras can be found only among those kites where both sets are at most countably infinite. Our results generalize analogous ones from \cite{DvKo} and, in many situations, the proofs are similar or inspired by original ones from \cite{DvKo}.

Let $K_{I,J}^{\lambda,\rho}(\mathbf{A})$ be a kite and $\alpha$ a cardinal. An element $\langle a_i\colon i \in I\rangle $ is said (i) $\alpha$-{\it dimensional} if $|\{i\in I\colon a_i \ne 1\}|=\alpha$, (ii) {\it finite-dimensional} if it is $\alpha$-dimensional for some finite cardinal $\alpha$. In the analogous way we say that $\langle \bar f_j\colon j \in J\rangle$ is $\alpha$-dimensional or finite-dimensional.

\begin{lemma}\label{le:4.1}
Let $K_{I,J}^{\lambda,\rho}(\mathbf{A})$ be a kite pseudo BL-algebra.

\noindent {\rm (1)}
$(A)^I$ is a maximal normal filter of $K_{I,J}^{\lambda,\rho}(\mathbf{A})$.

\noindent {\rm (2)} Let $F$ be a normal filter of $\mathbf A$. Then $F^I=\{\langle a_i\colon i \in I\rangle\colon$ for all  $i, a_i\in F\}$ and
$F^I_f,$ the system of finite-dimensional elements of $F^I,$ are normal
filters of $K_{I,J}^{\lambda,\rho}(\mathbf{A})$.
\end{lemma}

\begin{proof}
(1) Since the multiplication and ordering are in $A^I$ defined by coordinates, it is clear that $A^I$ is a filter. To show that $A^I$ is normal, we establish that it is closed under conjugations $y\ld xy$ by elements from $(\overline A)^J$. Thus if $x = \langle a_i\colon i \in I\rangle$ and $y = \langle \bar f_j\colon j \in J\rangle$, then $xy \in (\overline A)^J$ and $y\ld xy \in A^I.$ The same is true for right conjugates. To show that $A^I$ is maximal, take $y=\langle \bar f_j\colon j \in J\rangle$. Then the filter of the kite generated by $A^I$ and the element $y$, $F(A^I,y)$, is the set $\{a\in K_{I,J}^{\lambda,\rho}(\mathbf{A})\colon hy^n\le a$ for some integer $n \ge 0$ and $h \in A^I\},$ where $y^1 =1$ and $y^{n+1}=y^ny.$ We see that $F(A^I,y)= K_{I,J}^{\lambda,\rho}(\mathbf{A}).$

(2) Now it is clear that $F^I$ and $F^I_f$ are filters. To prove the normality, we use again left and right conjugates. Take an element
$\langle \bar f_j\colon j\in J\rangle$ and consider
$x=\langle \bar f_j\colon j\in J\rangle\ld
\langle a_i\colon i\in I\rangle\cdot\langle f_j\colon j\in J\rangle$ for
some $\langle a_i\colon i\in I\rangle\in F^I$. This is equal to $\langle b_i\colon i\in I\rangle$,
where
$$
b_i =\begin{cases}
f_{\rho^{-1}(i)}\rd (f_{\rho^{-1}(i)} \rd a_{\lambda(\rho^{-1}(i))}) &
\text{ if } \rho^{-1}(i) \text{ is defined}\\
e^{-1} & \text{ otherwise}.
\end{cases}\\
$$
By \cite[Lem 2.5(16)]{GLP}, we have $a_{\lambda(\rho^{-1}(i))} \le b_i$ so that $b_i\in F$ and consequently, $x \in F^I$. In the same way we can prove that $F^I$ is closed also under the right conjugates $\langle \bar f_j\colon j\in J\rangle\cdot
\langle a_i\colon i\in I\rangle\rd \langle f_j\colon j\in J\rangle$ which proves that $F^I$ is normal. In the analogous way we prove that also $F^I_f$ is a normal filter.
\end{proof}

The following statement shows that the subdirectly irreducible kites can appear only from subdirectly irreducible basic pseudo hoops.

\begin{lemma}\label{le:4.2}
If the kite pseudo BL-algebra $K_{I,J}^{\lambda,\rho}(\mathbf{A})$ is subdirectly irreducible, then $\mathbf A$ is subdirectly irreducible.

\end{lemma}

\begin{proof}
If $A=\{1\}$, the statement is evident. Now, let $\mathbf A$ be non-trivial. Suppose the converse, that is, $\mathbf{A}$ is not
subdirectly irreducible. We can find a family
$\{F_t\colon t\in t\}$ of non-trivial normal
filters of $\mathbf{A}$, such that $\bigcap_{t\in T}F_t = \{1\}$.
By Lemma~\ref{le:4.1}, $F_t^I$ is a normal filter of
$K_{I,J}^{\lambda,\rho}(\mathbf{A})$ for all $t\in T$. Suppose
$\langle a_i \colon i\in I\rangle$ belongs to $F_t^I$ for each
$t\in T$. Then, for any coordinate $k\in I$, we have that $a_k\in
F_t$ for all $t\in T$, which gives $a_k = 1$. Therefore,
$\bigcap_{t\in T}F_t^I = \{1\}$, showing that the set
$\{F_t^I\colon t\in T\}$ of non-trivial normal filters of
$K_{I,J}^{\lambda,\rho}(\mathbf{A})$ has the trivial intersection. Consequently,
$K_{I,J}^{\lambda,\rho}(\mathbf{A})$ is not subdirectly irreducible,
proving the lemma.
\end{proof}

For the next result we need the following notion originally introduced in \cite{DvKo}. We say that elements
$i,j\in I$ are {\it connected} if there is an integer $m \ge 0$ such that $(\rho\circ\lambda^{-1})^m(i)= j$ or $(\lambda\circ \rho^{-1})^m(i)= j$; otherwise, $i$ and $j$ are said to be {\it disconnected}.

The relation $i$ and $j$ are connected is an equivalence on $I$. We call this equivalence class a {\it connected component} of $I.$ We denote by $\mathcal C(I)$ the set of all connected components of $I.$

It is noteworthy to recall that if $C$ is a connected component of $I,$ then $\lambda^{-1}(C)= \rho^{-1}(C).$  Indeed, let $i \in C$ and $k=\lambda^{-1}(i).$ Then $j=\rho (k)= \rho \circ \lambda^{-1}(i) \in C.$ Hence, $k=\rho^{-1}(j)$ which proves $\lambda^{-1}(C)\subseteq \rho^{-1}(C).$ In the same way we prove the opposite inclusion. In particular, we have that $\lambda^{-1}(i)$ and $\rho^{-1}$ are defined for all $i \in I.$  In particular, we have $\lambda(\rho^{-1}(C))= \rho(\lambda^{-1}(C)) = C=\lambda(\lambda^{-1}(C))= \rho(\rho^{-1}(C)).$

The proof of the following theorem follows the basic steps from the proof of \cite[Thm 5.5]{DvKo}.

\begin{theorem}\label{th:4.3}
Let $\mathbf{A}$ be a non-trivial basic pseudo hoop and let
$K_{I,J}^{\lambda,\rho}(\mathbf{A})$ be a kite.
The following are equivalent:
\begin{enumerate}
\item $\mathbf{A}$ is subdirectly irreducible and for all $i,j\in I$ there exists an integer
  $m\ge 0$ such that
$(\rho\circ\lambda^{-1})^m(i) = j$ or
$(\lambda\circ\rho^{-1})^m(i) = j$.
\item $K_{I,J}^{\lambda,\rho}(\mathbf{A})$ is a subdirectly irreducible pseudo BL-algebra.
\end{enumerate}
\end{theorem}

\begin{proof}
(1) $\Rightarrow$ (2) Let $N$ be the smallest
non-trivial  normal filter of $\mathbf{A}$. By
Lemma~\ref{le:4.1}, we have that $N^I_f,$ the system of finite-dimensional
elements of $N^I,$ is a normal filter of
$K_{I,J}^{\lambda,\rho}(\mathbf{A}).$
It is necessary to show that
$N^I_f$ is the smallest non-trivial filter of the kite. Since for any $a\in
N^I_f\setminus\{1\}$ there is a one-dimensional element $a'$ with
$a\leq a'<1$, it suffices to prove that any one-dimensional element
$b\in N^I_f\setminus\{1\}$ generates $N^I_f$. Without loss of generality
assume $b = \langle b_0, 1,\dots\rangle$; this is always
achievable by a suitable re-ordering of $I$, regardless of its
cardinality. Observe that $b_0$ generates $N$, since $N$ is the
smallest non-trivial normal filter of $\mathbf{A}$. It follows that
$b$ generates all members of $N^I_f$ of the form $\langle a,
1,\dots\rangle$, using only conjugates of the same form.
Consider an arbitrary $i\in I$. By assumption, there is an integer $m\ge 0$
with $(\rho\circ\lambda^{-1})^m(0) = i$ or $(\lambda\circ\rho^{-1})^m(0) =
i$.  We denote be $a^{\sim\sim^m}$ and $a^{--^m}$ the $m$-times performing double negations of the same kind. An easy calculation shows that for an element $u = \langle
a, 1,\dots\rangle$, one of the following must be the case:
\begin{itemize}
\item if $(\rho\circ\lambda^{-1})^m(0) = i$, then
$u^{\sim\sim^m}= \langle
1,\dots,1,a,1,\dots\rangle$,
\item if $(\lambda\circ\rho^{-1})^m(0) = i$, then
$u^{--^m} = \langle
1,\dots,1,a,1,\dots\rangle.$
\end{itemize}
Re-numbering $I$ if necessary, we may assume that $a$ occurs in the $m$-the co-ordinate.
By taking appropriate finite meets, it follows that every element of $N^I_f$
can be generated, which proves the implication.

(2) $\Rightarrow$ (1)
By Lemma~\ref{le:4.2}, we can assume
$\mathbf{A}$ is subdirectly irreducible. Then, suppose there are
$i,j\in I$ such that for all $m\in \mathbb N$ we have
$(\rho\circ\lambda^{-1})^m(i) \neq j$ and $(\lambda\circ\rho^{-1})^m(i) \neq
j$. Now, let $I_0$ and $I_1$ be
connected components of $I$ such that $i\in I_0$ and $j\in I_1$.
Clearly, $I_0$ and $I_1$ are disconnected, that is, no member of
$I_0$ is connected to any member of $I_1$. We will prove that
$N^{I_0}\cap N^{I_1} = \{1\}$, from which it follows immediately
that $K_{I,J}^{\lambda,\rho}(\mathbf{A})$ is not subdirectly
irreducible. In fact, it suffices to show that for an element $u =
\langle u_i\colon i\in I\rangle$ such that $u_i =1$
for all $i\notin I_0$, and for any element $b$, the conjugate $b\ld
ub$ has $(b\ld ub)(i) = 1$ if $i\notin I_0$, and the same holds
for $bu\rd b$. Take $b = \langle \bar b_j\colon j\in J\rangle$. We have
\begin{align*}
b\ld ub &=
\langle \bar b_j\colon j\in J\rangle\ld
\langle u_i\colon i\in I\rangle\cdot
\langle \bar b_j\colon j\in J\rangle\\
&= \langle \bar b_j\colon j\in J\rangle\ld
\langle \overline {b_j\rd u_{\lambda(j)}}\colon j\in J\rangle\\
&= \langle c_i\colon i\in I\rangle,
\end{align*}
where  $c_i =\begin{cases}
b_{\rho^{-1}(i)} \rd
(b_{\rho^{-1}(i)} \rd u_{\lambda(\rho^{-1}(i))})&
\text{ if } \rho^{-1}(i) \text{ is defined}\\
1 & \text{ otherwise.}
\end{cases}$\\
Now, by assumption $u_i = 1$ for $i\notin I_0$, and
by connectedness, $\lambda(\rho^{-1}(i))\notin I_0$ if $i\notin I_0$.
Therefore, $c_i$ can be different from $1$ only if
$i\in I_0$, and thus $b\ld ub$ is of the required form. The claim for the other conjugate follows by symmetry.
\end{proof}

The proof of the following lemma is identical with the proof of \cite[Lem 5.6]{DvKo}, therefore, we omit it here.

\begin{lemma}\label{le:4.4}
If $\mathbf{A}$ is a non-trivial basic pseudo hoop,
$K_{I,J}^{\lambda,\rho}(\mathbf{A})$ a subdirectly irreducible kite,
and $I$ and $J$ are finite, then $K_{I,J}^{\lambda,\rho}(\mathbf{A})$
is isomorphic to one of:
\begin{enumerate}
\item $K_{0,0}^{\emptyset,\emptyset}(\mathbf{A})$,
$K_{1,1}^{id,id}(\mathbf{A})$, $K_{1,0}^{\emptyset,\emptyset}(\mathbf{A})$,
\item $K_{n,n}^{\lambda,\rho}(\mathbf{A})$, for $n>1$, with $\lambda(j) = j$ and
$\rho(j) = j+1\ (\mathrm{mod}\ n)$,
\item $K_{n+1,n}^{\lambda,\rho}(\mathbf{A})$, for $n>1$, with $\lambda(j) = j$ and
$\rho(j) = j+1$.
\end{enumerate}
\end{lemma}

In what follows, we show that for $I$ and $J$ infinite, the subdirect irreducibility of the kite $K_{I,J}^{\lambda,\rho}(\mathbf{A})$ implies $I$ is at most countable.

\begin{lemma}\label{countable-dim}
Let  $K_{I,J}^{\lambda,\rho}(\mathbf{A})$ be a subdirectly
irreducible kite of a non-trivial basic pseudo hoop $\mathbf A$. Then $I$ and $J$ are at most countably infinite.
\end{lemma}

\begin{proof}
First, observe that if $I\setminus \bigl(\lambda(J)\cup\rho(J)\bigr)$ is
nonempty, then any $j\in \lambda(J)\cup\rho(J)$  is
disconnected from any $i\in I\setminus \bigl(\lambda(J)\cup\rho(J)\bigr)$.
Therefore, $I = \lambda(J)\cup\rho(J)$. It follows that $I$ is
countable iff $J$ is. Suppose $I$ and $J$ are uncountable and pick an
$i\in I$. Consider the set
$P(i) = \{(\rho\circ\lambda^{-1})^n(i)\colon n\ge 0\}\cup
\{(\lambda\circ\rho^{-1})^n(i)\colon n\ge 0\}$. Clearly $P(i)$ is
at most countable; so there is a $j\in I\setminus P(i)$.
But $P(i)$ exhausts all finite paths of back-and-forth
alternating $\lambda$ and $\rho$ beginning from $i$. Then, $i$ and $j$ are
disconnected, contradicting Theorem~\ref{th:4.3}.
\end{proof}

Now we present two consequences; for their detailed proofs see \cite[Lem 5.8, Thm 5.9]{DvKo}.

\begin{lemma}\label{countable-dim-maps}
Let  $K_{I,J}^{\lambda,\rho}(\mathbf{A})$ be a subdirectly
irreducible kite of non-trivial $\mathbf A$ with countably infinite $I$ and $J$.
Then, one of the following three cases must obtain:
\begin{enumerate}
\item $\lambda$ and $\rho$ are bijections.
\item $\lambda$ is a bijection and $|I\setminus\rho(J)| = 1$.
\item $\rho$ is a bijection and $|I\setminus\lambda(J)| = 1$.
\end{enumerate}
\end{lemma}

We present a classification of subdirectly irreducible kite pseudo BL-algebras.

\begin{theorem}\label{classif}
Let $K_{I,J}^{\lambda,\rho}(\mathbf{A})$ be a subdirectly
irreducible kite of a non-trivial $\mathbf A$.  Then $K_{I,J}^{\lambda,\rho}(\mathbf{A})$ is isomorphic to
precisely one of:
\begin{enumerate}
\setcounter{enumi}{-1}
\item $K_{0,0}^{\emptyset,\emptyset}(\mathbf{A})$,
$K_{1,1}^{id,id}(\mathbf{A})$, $K_{1,0}^{\emptyset,\emptyset}(\mathbf{A})$,
\item $K_{n,n}^{\lambda,\rho}(\mathbf{A})$, with $\lambda(j) = j$ and
$\rho(j) = j+1\ (\mathrm{mod}\ n)$.
\item $K_{\mathbb{Z},\mathbb{Z}}^{\lambda,\rho}(\mathbf{A})$, with $\lambda(j) = j$ and
$\rho(j) = j+1$.
\item $K_{\omega,\omega}^{\lambda,\rho}(\mathbf{A})$, with $\lambda(j) = j$ and
$\rho(j) = j+1$.
\item $K_{\omega,\omega}^{\lambda,\rho}(\mathbf{A})$, with $\lambda(j) = j+1$ and
$\rho(j) = j$.
\item $K_{n+1,n}^{\lambda,\rho}(\mathbf{A})$, with $\lambda(j) = j$ and
$\rho(j) = j+1$.
\end{enumerate}
Moreover, types {\rm (1)} and {\rm (2)} consist entirely of pseudo
MV-algebras, and the other types contain no pseudo MV-algebras except
the two-element Boolean algebra. A kite of type {\rm (3)} or {\rm
(4)} is good if and only if it is a two-element Boolean algebra. A
kite of type {\rm (5)} is good if and only if $J=\emptyset$.
\end{theorem}

We notice that by Theorem \ref{th:3.3}, the last theorem says that there are also subdirectly irreducible kites that are not good. In addition, it shows that there are subdirectly irreducible kites that are not linear.

\begin{remark}\label{re:1}
{\rm If $\mathbf A$ is trivial, the corresponding kite is a two-element Boolean algebra, therefore, Theorem \ref{th:3.3}, Lemmas \ref{le:4.4}--\ref{countable-dim-maps} and Theorem \ref{classif} are not necessarily valid.
}
\end{remark}

The last result in this section is a Birkhoff's Subdirect Representation type theorem saying that each kite is a subdirect product of subdirectly irreducible ones.

\begin{lemma}\label{subd-repr-gives-subd-repr}
Let $\mathbf{A}$ be a basic pseudo hoop, subdirectly represented as
$\mathbf{A}\leq \prod_{s\in  S}\mathbf{A}_s$. Then
$K_{I,J}^{\lambda,\rho}(\mathbf{A})$ is subdirectly represented as
$K_{I,J}^{\lambda,\rho}(\mathbf{A})\leq
\prod_{s\in S}K_{I,J}^{\lambda,\rho}(\mathbf{A}_s)$.
\end{lemma}

\begin{proof}
It is an easy consequence of Lemma \ref{le:4.1}.
\end{proof}

\begin{lemma}\label{le:4.9}
Let $K^{\lambda,\rho}_{I,J}(\mathbf A)$ be a kite pseudo BL-algebra of a basic pseudo hoop  $\mathbf A$. Then
$K^{\lambda,\rho}_{I,J}(\mathbf A)$ is a subdirect
product of the system of kite pseudo effect algebras $(K_{I',J'}^{\lambda',\rho'}(\mathbf A)\colon I'\in \mathcal C(I))$, where $I'$ is any
connected component of $I$, $J'=\lambda^{-1}(I')=\rho^{-1}(I'),$ and $\lambda',\rho'\colon J'\to I'$ are the restrictions of $\lambda$ and $\rho$ to $J'\subseteq I.$
\end{lemma}

\begin{proof}
We see that $\lambda',\rho'\colon J'\to I'$ are injections.
Let $I'$ be a connected component of $I$.
Let $N_{I'}$ be the set of all elements
$f = \langle \bar f_j: j \in J\rangle$
such that  $f_j = 1$ whenever $j\in J'$.
It is straightforward to see that $N_{I'}$ is a normal ideal
of $K_{I,J}^{\lambda,\rho}(\mathbf A)$.

It is also not difficult
to see that $K_{I,J}^{\lambda,\rho}(\mathbf A)/N_{I'}$ is isomorphic to
$K_{I',J'}^{\lambda',\rho'}(\mathbf A)$.

As connected components are disjoint, we have
$\bigcap_{I'\in \mathcal C(I)} N_{I'} = \{1\}.$ This proves $K_{I,J}^{\lambda,\rho}(\mathbf A)= K_{I,I}^{\lambda,\rho}(\mathbf A)\leq \prod_{I'\in \mathcal C(I)} K_{I',J'}^{\lambda',\rho'}(\mathbf A).$
\end{proof}

\begin{theorem}\label{th:7.5}
Every kite pseudo effect algebra of a basic pseudo hoop is  a subdirect product of subdirectly irreducible kite pseudo effect algebras.
\end{theorem}

\begin{proof}
If $\mathbf A$ is a trivial pseudo hoop, the statement is trivial. Take a kite $K_{I,J}^{\lambda,\rho}(\mathbf A)$ for a non-trivial $\mathbf A$.  If the kite is not subdirectly irreducible, by Theorem \ref{th:4.3}, there are two possible cases:
(i) $\mathbf A$ is not subdirectly irreducible, or
(ii) $\mathbf A$ is subdirectly irreducible but
there exist $i,j\in I$ such that, for every integer
$m\ge 0,$ we have $(\rho\circ\lambda^{-1})^m(i)\neq j$ and
$(\lambda\circ\rho^{-1})^m(i)\neq j$. Observe that this happens if and only if
$i$ and $j$ do not belong to the same connected component of $I$.

By Lemma \ref{subd-repr-gives-subd-repr},
we can reduce (i) to (ii). So, suppose $\mathbf A$ is subdirectly irreducible.
Then, using Lemma~\ref{le:4.9}, we can subdirectly embed
$K_{I,J}^{\lambda,\rho}(\mathbf A)$ into
$\prod_{I'} K_{I',J'}^{\lambda',\rho'}(\mathbf A)$, where $I'$
ranges over the connected components of $I,$ $J'=\lambda^{-1}(I')$ and $\lambda', \rho'$ are restrictions of $\lambda, \rho$ to $J'.$ But then each
$K_{I',J'}^{\lambda,\rho}(\mathbf A)$ is subdirectly irreducible by
Theorem \ref{th:4.3}.
\end{proof}

Now we present a simple but important consequence of Theorem \ref{th:7.5}.

\begin{corollary}\label{si-kites-generate}
Let\/ $\mathsf{K}$ be the variety generated by all kites.
Then $\mathsf{K}$ is generated by all subdirectly irreducible kites.
\end{corollary}

Theorem \ref{th:7.5} and Corollary \ref{si-kites-generate} can be strengthened in the same way as in \cite[Thm 6.5, Cor 6.6]{DvKo}. Since their proofs follow the same basic steps as those in \cite{DvKo}, we recommend to consult \cite{DvKo} for details of their proofs, if necessary.

Let $\mathbf A$ be a basic pseudo hoop. We introduce the following notions.
A kite $K_{I,J}^{\lambda,\rho}(\mathbf{A})$ will be called
\emph{$n$-dimensional} if $|I| = n\in\mathbb N$, and
\emph{finite-dimensional} if it is $n$-dimensional for some $n$. We
write $\mathcal{K}_n$ for the class of all $n$-dimensional kites and
$\mathsf{K_n}$ for the variety generated by $\mathcal{K}_n$.

\begin{theorem}\label{generation}
The variety $\mathsf{K}$ is generated by all finite-dimensional kites.
\end{theorem}

\begin{corollary}
The variety $\mathsf{K}$ generated by all kites is the varietal join of
varieties $\mathsf{K}_n$ generated by $n$-dimensional kites. Briefly,
$$
\mathsf{K} = \bigvee_{n=0}^\infty \mathsf{K}_n.
$$
\end{corollary}
Finally, we formulate two problems.  Let $\mathcal V$ be a family of bounded pseudo hoops, $\lambda,\rho: J \to I$ be given. (1) Describe the variety $\mathsf V^{\lambda,\rho}_{I,J}(\mathcal V)$ of pseudo BL-algebras generated by the system $\{K^{\lambda,\rho}_{I,J}(\mathbf A)\colon \mathbf A \in \mathcal V\}.$  For example, if $I=\emptyset,$ then by Theorem \ref{classif},  $\mathsf V^{\lambda,\rho}_{I,J}(\{\mathbf A\})$ is the variety of Boolean algebras for any basic pseudo hoop $\mathbf A$.

(2) Describe the variety generated by all kites corresponding to all basic pseudo hoops $\mathbf A \in \mathcal V.$

\section{Conclusion}

We have presented a new construction of pseudo BL-algebras starting with a basic pseudo hoop, two sets, $I$ and $J$, and with two injections $\lambda,\rho: J \to I.$ This construction, Theorem \ref{th:3.1}, is a generalization of that from \cite{DvKo}, where the starting point was a special kind of a basic pseudo hoop~-- the negative cone of an $\ell$-group. It is interesting to note that even from a commutative hoop, we can obtain a non-commutative pseudo BL-algebras. We have characterized cases when the resulting algebra is a pseudo BL-algebra or a good pseudo BL-algebra, Theorem \ref{th:3.3}.

We have classified subdirectly irreducible kites, Theorem \ref{classif}, and we have shown that subdirectly irreducible kites can be found only among those kites with $J$ and $I$ at most countable, Lemma \ref{countable-dim}. We have characterized subdirectly irreducible kites with respect to subdirect irreducibility of the starting basic pseudo hoop, Theorem \ref{th:4.3}, and we have shown an analogue of the Birkhoff Subdirect Representation Theorem for kites,  Theorem \ref{th:7.5}.

This paper gives a method of constructing a new large class of interesting examples of kite pseudo BL-algebras and opens a new window into theory of pseudo BL-algebras.


\begin{thebibliography}{DvVe2}

\bibitem[AgMo]{AgMo}
P. Aglian\`o, F. Montagna, {\it Varieties of BL-algebras I: General properties}, J. Pure Appl. Algebra {\bf 181} (2003), 105-–129.

\bibitem[Bos1]{Bos1}
B. Bosbach, {\it Komplement\"are Halbgruppen. Axiomatik und Arithmetik}, Fund. Math. {\bf 64} (1966), 257-–287.

\bibitem[Bos2]{Bos2}
B. Bosbach, {\it Komplement\"are Halbgruppen. Kongruenzen and Quotienten}, Fund. Math. {\bf 69} (1970), 1-–14.

\bibitem[DGI1]{DGI1}
A. Di Nola, G. Georgescu, A. Iorgulescu, \emph{Pseudo-BL algebras
I}, Multiple Val. Logic {\bf 8} (2002), 673--714.


\bibitem[DGI2]{DGI2}
A. Di Nola, G. Georgescu, A. Iorgulescu, \emph{Pseudo-BL algebras
II}, Multiple Val. Logic {\bf 8} (2002), 715--750.

\bibitem[Dvu1]{Dvu1} A. Dvure\v censkij,    {\it  Pseudo MV-algebras
are intervals in $\ell$-groups}, J. Austral. Math. Soc. {\bf 72}
(2002), 427--445.

\bibitem[Dvu2]{Dvu2}
A. Dvure\v{c}enskij, {\it Aglian\`o--Montagna type decomposition of linear
pseudo hoops and its applications,} J. Pure Appl. Algebra {\bf 211}
(2007), 851--861.

\bibitem[DGK]{DGK}
A. Dvure\v{c}enskij, R. Giuntini,  T. Kowalski, {\it  On the
structure of pseudo BL-algebras and pseudo hoops in quantum logics},
Found. Phys. \textbf{40} (2010), 1519--1542.
DOI:10.1007/s10701-009-9342-5

\bibitem[DvKo]{DvKo}
A. Dvure\v censkij, T. Kowalski, {\it Kites and pseudo BL-algebras}, Algebra Univer. (to appear). http://arxiv.org/submit/510377

\bibitem[EsGo]{EsGo}
F. Esteva, L. Godo, {\it Monoidal t-norm based logic: Towards a logic of left-continuous t-norms}, Fuzzy Sets and Systems {\bf 124} (2001), 271-–288.

\bibitem[GaTs]{GaTs}
N.  Galatos, C. Tsinakis, {\it Generalized MV-algebras}, J. Algebra {\bf 283} (2005), 254--291.

\bibitem[GeIo]{GeIo} G. Georgescu, A. Iorgulescu,
{\it Pseudo-MV algebras},  Multi Valued Logic {\bf 6} (2001),
95--135.


\bibitem[GLP]{GLP} G. Georgescu, L. Leu\c{s}tean, and V. Preoteasa,
{\it Pseudo-hoops,}  J.   Multiple-Val. Logic and Soft Computing
{\bf 11} (2005), 153--184.

\bibitem[Haj1]{Haj1}
P. H\'ajek, {\it Basic fuzzy logic and BL-algebras}, Soft Comput. {\bf 2} (1998), 124-–128.

\bibitem[Haj2]{Haj2}
P. H\'ajek, {\it Fuzzy logics with noncommutative conjunctions}, J. Logic Comput. {\bf 13} (2003), 469-–479.

\bibitem[Haj3]{Haj3}
P. H\'ajek, {\it Observations on non-commutative fuzzy logic}, Soft Comput. {\bf 8} (2003), 38-–43.


\bibitem[JiMo]{JiMo}
P. Jipsen, F.   Montagna, {\it  On the
structure of generalized BL-algebras}, Algebra Univer. \textbf{55} (2006), 226--237.

\bibitem[Rac]{Rac} J. Rach\r unek, {\it  A
non-commutative generalization of MV-algebras}, Czechoslovak Math.
J. {\bf 52} (2002), 255--273.


\end{thebibliography}
\end{document}